\documentclass{amsart}
\usepackage[utf8]{inputenc}
\usepackage{amsmath,pdfsync,verbatim,graphicx,epstopdf,enumerate,xcolor}
\usepackage[colorlinks=true]{hyperref}
\hypersetup{allcolors=red}
\usepackage{cite}
\usepackage{xcolor}
\DeclareMathOperator\supp{supp}
\numberwithin{equation}{section}
\newcommand{\R}{\mathbb{R}}
\newcommand{\D}{\mathrm{d}}
\newtheorem{theorem}{Theorem}

\newtheorem{definition}{Definition}[section]
\newtheorem{assumption}{Assumption}

\newtheorem{lemma}{Lemma}[section]

\newtheorem{proposition}{Proposition}[section]
\newtheorem{remark}{Remark}[section]
\title[Local data inverse problem for the polyharmonic operator]{Local data inverse problem for the polyharmonic operator with anisotropic perturbations}

\author[Bhattacharyya and Kumar]{Sombuddha Bhattacharyya* and Pranav Kumar**}
\address{*Department of Mathematics, Indian Institute of Science Education and Research, Bhopal.
\newline
E-mail:{\tt \ sombuddha@iiserb.ac.in}}
\address {**Department of Mathematics, Indian Institute of Science Education and Research, Bhopal.
\newline
E-mail:{\tt \ pranav19@iiserb.ac.in}}

\begin{document}

\begin{abstract}
In this article, we study an inverse problem with local data for a linear polyharmonic operator with several lower order tensorial perturbations. 
We consider our domain to have an inaccessible portion of the boundary where neither the input can be prescribed nor the output can be measured. We prove the unique determination of all the tensorial coefficients of the operator from the knowledge of the Dirichlet and Neumann map on the accessible part of the boundary, under suitable geometric assumptions on the domain.
\end{abstract}

\subjclass [2020]{Primary 35R30, 31B20, 31B30, 35J40.}
\keywords {Calder\'{o}n problem, polyharmonic operator, Anisotropic perturbation, Tensor tomography, Momentum ray transform.}

\maketitle

\section{Introduction}
Let $\Omega$ be an open, bounded domain in $\R^{n}$ with smooth boundary $\partial\Omega$, where $n\geq3$. We consider the perturbed polyharmonic operator $\mathcal{L}(x,D)$ of order $2m$ with perturbations up to order $m$ of the form:
\begin{equation}\label{Operator}
\mathcal{L}(x, D):= (-\Delta)^m + \sum_{j=0}^{m-1} \sum_{i_{1}, \cdots, i_{m-j}=1}^{n} A_{i_{1} \cdots i_{m-j}}^{m-j}(x) D_{i_{1} \cdots  
 i_{m-j}}^{m-j}+q(x); \quad m\geq 2,
\end{equation}
where $D_{i_{1} \dots i_{j}}^{j}=\frac{1}{\mathrm{i}^{j}} \frac{\partial^{j}}{\partial x_{i_{1}\dots} \partial x_{i_{j}}}$, $i_1, \dots, i_j\in \{1,2,\dots, n\}$, $1\leq j\leq m$, and the perturbations coefficients $A_{i_{1} \cdots i_{j}}^{j} = A^{j}\in C_{c}^{\infty}(\Omega, \mathbb{C}^{n^{j}})$ symmetric tensor fields of order $j$  for $1 \leq j \leq m$ and $q(x) \in L^{\infty}(\Omega, \mathbb{C})$.
In this article, we solve a partial data inverse problem related to the operator  \eqref{Operator}. We prove the unique recovery of all the lower order perturbations $A^{j} \text { for } 1\leq j \leq m$,  and $q$ 
in $\Omega$ from the boundary data measured only on a part $\Gamma$ of the boundary $\partial\Omega$ while the remaining part $\partial \Omega \setminus \Gamma$ is part of a hyperplane. 

Inverse problems for recovering unknown coefficients of PDEs were pioneered by Alberto Calder\'on in his work \cite{Calderon_Paper} on electrical impedance tomography. 
The groundbreaking work by Sylvester and Uhlmann in \cite{sylvester1987global} on the inverse problem for the conductivity equation inspired subsequent developments. 
For a comprehensive survey of Calder\'on-type inverse problems, see \cite{uhlmann2009electrical}.
The study of coefficient recovery problems for biharmonic and polyharmonic operators was initiated by Krupchyk, Lassas, and Uhlmann in \cite{krupchyk2012determining,krupchyk2014inverse}. 
It was soon followed by \cite{Yang-Paper,ghosh2016determination,BG_19,bhattacharyya2021unique,SS_linearize_polyharmonic}, where coefficient recovery for biharmonic and polyharmonic operators in a bounded domain with dimensions $n\geq 3$ have been discussed.
The higher order elliptic operators, as represented by \eqref{Operator}, have received significant attention in various areas, such as physics and geometry, due to their connection with the theory of elasticity equation on thin domains given by the Kirchhoff plate equation (perturbed biharmonic operator) and the study of the Paneitz-Branson equation in conformal geometry. 
For more applications of biharmonic and polyharmonic operators, see \cite{selvadurai2000partial,Polyharmonic_Book}.

To describe our setup, let us define the following notations.
Define $\mathcal{D}(\mathcal{L}(x,D)), $ where,
\[
\mathcal{D}(\mathcal{L}(x, D)):=\Big{\{}u \in H^{2 m}(\Omega): u|_{\partial\Omega}=(-\Delta )u|_{\partial\Omega} =\cdots =(-\Delta)^{m-1}u|_{\partial\Omega} =0\Big{\}.}
\]
The operator $\mathcal{L}(x, D)$ with the domain $\mathcal{D}(\mathcal{L}(x, D))$ is an unbounded closed operator on $L^{2}(\Omega)$ with discrete spectrum \cite{Gerd_Grubb}.
\begin{assumption}\label{Assumption-1}
We assume that $0$ is not an eigenvalue of the operator $\mathcal{L}(x, D): \mathcal{D}(\mathcal{L}(x, D)) \rightarrow L^{2}(\Omega)$.
\end{assumption}
Under Assumption \ref{Assumption-1}, the following boundary value problem with Navier boundary conditions:
\begin{equation}\label{Direct Problem}
    \begin{aligned}
      \mathcal{L}(x, D) u &=0 \quad \text { in } \Omega ,\\
 \left(u|_{\partial\Omega},(-\Delta )u|_{\partial\Omega}, \cdots , (-\Delta)^{m-1}u|_{\partial\Omega}\right) &=\left(f_{0}, f_{1}, \cdots, f_{m-1}\right),
    \end{aligned}
\end{equation}
has a unique solution $u \in H^{2 m}(\Omega)$ for any $f :=\left(f_{0}, f_{1}, \cdots, f_{m-1}\right) \in \prod_{i=0}^{m-1} H^{2m-2i-\frac{1}{2}}(\partial \Omega)$.
We state the existence and uniqueness of a solution to the problem \eqref{Direct Problem}.

\begin{theorem}[\cite{tanabe2017functional} Lemma 5.13]
Let the coefficients $A^{j}$ for $1 \leq j \leq m$ belong to $C_{c}^{\infty}(\Omega, 
\mathbb{C}^{n^{j}})$ and $q(x) \in L^{\infty}(\Omega, \mathbb{C})$ in \eqref{Operator}. Suppose Assumption \ref{Assumption-1} and $f=\left(f_{0}, f_{1}, \ldots, f_{m-1}\right) \in \prod_{i=0}^{m-1} H^{2 m-2 i-\frac{1}{2}}(\partial \Omega)$. Then there exists a unique solution $u\in H^{2m}(\Omega)$ to the problem  \eqref{Direct Problem} satisfying 
\[
\|u\|_{H^{2m}(\Omega)} \leq C\|f\|_{\prod_{i=0}^{m-1}H^{2m-2i-\frac{1}{2}}(\partial \Omega)},
\]
where $C>0$ is a constant.
\end{theorem}

We define the Dirichlet-to-Neumann (D-N) map by
\[
\mathcal{N}: \prod_{i=0}^{m-1} H^{2m-2i-\frac{1}{2}}(\partial \Omega) \rightarrow \prod_{i=0}^{m-1} H^{2 m-2i-\frac{3}{2}}(\partial \Omega),
\]
\[
  \mathcal{N}(f)=\left(\partial_{\nu} u_f|_{\partial \Omega}, \cdots,\partial_{\nu}(-\Delta)^{k}u_f|_{\partial \Omega}, \cdots,\partial_{\nu}(-\Delta)^{m-1} u_f|_{\partial \Omega}\right),  \quad f=(f_0,\cdots,f_{m-1}); 
\]
where $u_f \in H^{2m}(\Omega)$ is the unique solution to the problem \eqref{Direct Problem} corresponding to $f \in \prod_{i=0}^{m-1} H^{2m-2i-\frac{1}{2}}(\partial \Omega)$ and $\nu$ is the outward unit normal vector to the boundary $\partial \Omega$.
Let $\Gamma$ be a non-empty open subset of the boundary $\partial \Omega$, we define the space 
\[
\mathcal{H}_\Gamma := \{ f \in \prod_{i=0}^{m-1} H^{2m-2i-\frac{1}{2}}(\partial\Omega) : \supp{f} \subset \Gamma\}.
\]
We define the partial D-N map as
\begin{align}\label{Partial_DN_Map}
    &\mathcal{N}_{\Gamma}: \mathcal{H}_{\Gamma} \rightarrow \prod_{i=0}^{m-1} H^{2 m-2i-\frac{3}{2}}(\Gamma),
\end{align}
\[\mbox{as }\quad \mathcal{N}_{\Gamma}(f):=\left(\partial_{\nu} u_f|_{\Gamma}, \partial_{\nu}(-\Delta)u_f|_{\Gamma}, \cdots,\partial_{\nu}(-\Delta)^{m-1} u_f|_{\Gamma}\right).\]

We prove unique recovery of the coefficients of \eqref{Operator} from the partial D-N map $\mathcal{N}_{\Gamma}$, assuming that the inaccessible part of the boundary $\partial \Omega\setminus\Gamma$ is a part of a hyperplane. 

\subsection{Statement of the main result}
Throughout the article, we follow the standard tensor notations used in \cite{Sharafutdinov_book}.
Given an integer $m\geq 0$, we write $T^{m}=T^{m}(\R^{n})$ to denote the complex vector space of $\R$-multilinear functions from $\underbrace{\R^{n}\times \cdots \times \R^{n}}_\text{m-times}$ to $\mathbb{C}$. Let $e_{1}, \cdots, e_{n}$ be a basis for $\R^{n}$. We write the components of the tensor $f \in T^{m}$ as $f_{i_{1}\cdots i_{m}} := f(e_{i_{1}},\cdots,e_{i_{m}})$.
The notation $S^{m}= S^{m}(\R^{n})$ denotes the subspace of $T^{m}$ that consists of functions symmetric in all their arguments over the real vector space $\R^{n}$.

For $f\in T^{m}$, $g\in T^{k}$, the tensor product $f\otimes g\in T^{m+k}$ is defined by 
\[
f\otimes g(x_{1}, \cdots, x_{m+k})=f(x_{1},\cdots, x_{m})g(x_{m+1},\cdots, x_{m+k}).
\]
Let $\sigma: T^{m}\to S^{m}$ be the symmetrization defined by
\[
\sigma f(x_{1}, \cdots, x_{m})=\frac{1}{m!}\sum_{\pi\in \Pi_{m}}f(x_{\pi(1)}, \cdots, x_{\pi(m)}),
\]
where $\Pi_{m}$ is the set of all permutations of the set $\{1,\cdots, m\}$.

Given $f\in T^{m}, g\in T^{k}$, the symmetrized tensor product $u\odot v$ defined by 
\[\begin{aligned}
(f\odot g)(x_{1}, \cdots x_{m+k})
:=&\sigma(f \otimes g)(x_{1}, \cdots x_{m+k}) \\
=&\frac{1}{(m+k)!}\sum_{\pi\in\Pi_{m+k}}f(x_{\pi(1)},\cdots, x_{\pi(m)})g(x_{i_{\pi(m+1)}}, \cdots, x_{i_{\pi(m+k)}}).
\end{aligned}\]
We define the operator $i_{\delta}: S^{m}\to S^{m+2}$ as follows:
\[(i_{\delta}f)_{i_{1}\cdots i_{m+2}}
:= (\delta \odot f)_{i_{1}\cdots i_{m+2}}
= \sigma(f_{i_{1}\cdots i_{m}}\otimes \delta_{i_{m+1}i_{m+2}}).
\]

\begin{definition}\label{Def_1}
We say a symmetric tensor $A^j \in S^j(\R^n)$, $j \geq 2$, is partial isotropic if there exists a symmetric $(j-2)$-tensor field  $\overline{A}^{j-2} \in S^{j-2}(\R^n)$ in $\Omega$ such that $A^{j}=i_{\delta} \overline{A}^{j-2}=\overline{A}^{j-2} \odot \delta$, where $\delta$ is the Kronecker delta tensor and $\odot$ denotes symmetrized product of tensors.  
\end{definition}
As an example, any $n$ dimensional partially isotropic $2$-tensors (matrix) are of the form $a(x)I_n$ where
$a(x)$ is a function ($0$-tensor) and $I_n$ is the $n \times n$ identity matrix.
We assume the leading order perturbation $A^{m}$ of the operator \eqref{Operator} is to be partially isotropic, which means $A^{m} = i_{\delta}\overline{A}^{m-2}$ for some $m-2$ order symmetric tensor field $\overline{A}^{m-2}$.
We now state the main result of this article.
\begin{theorem}\label{Main_Result} Let $\Omega \subset \left\{\mathbb{R}^{n};  x_{n}>0\right\}$ be an open bounded domain with smooth boundary, where $n\geq3$. Let $\Gamma_{0} = \partial\Omega \cap \left\{x_{n}=0\right\}$ be non-empty, open and let $\Gamma = \partial \Omega \setminus 
 \Gamma_{0} $. Let $A^{j},\widetilde{A}^{j} \in C_c^{\infty}
(\Omega) \text{ for } 1\leq j \leq m$  and $q,  \widetilde{q} \in L^{\infty}(\Omega)$. We consider a new operator $\widetilde{\mathcal{L}}(x, D)$ by replacing $A^j$ by $\tilde{A}^j$ and $q$ by $\tilde{q}$ 
in \eqref{Operator}. Let $\mathcal{N}_{\Gamma}$ and $\widetilde{\mathcal{N}}_{\Gamma}$ be the partial D-N map corresponding to the operators $\mathcal{L}(x, D)$ and $\widetilde{\mathcal{L}}(x, D)$ respectively.
Let the coefficients be such that Assumption \ref{Assumption-1} is satisfied for $\mathcal{L}(x, D)$ and $\widetilde{\mathcal{L}}(x, D)$. We further assume that for $j=m$, $A^{m}, \widetilde{A}^{m}$ are partially isotropic. If 
\begin{equation}
\mathcal{N}_{\Gamma}(f)= \widetilde{\mathcal{N}_{\Gamma}}(f), \quad \text { for all } f \in \prod_{i=0}^{m-1} H^{2 m-2i-\frac{1}{2}}(\partial \Omega) \text { with } \supp (f)\subset \Gamma,
\end{equation}
then for $1\leq j\leq m$, we have
\[A^{j} = \widetilde{A}^{j} \quad\text { and } \quad q=\widetilde{q} \quad \text { on }\overline{\Omega}.\]
\end{theorem}

In the last few decades, partial data inverse problems, in which the boundary data is available only on a part of the boundary, has attracted considerable attention, see; \cite{KS_Survey}. For dimensions $n\geq 3$, Bukhgeim and Uhlmann in \cite{bukhgeim2002recovering} established the first uniqueness result for the Schr\"odinger operator $(-\Delta + q)$ in a bounded domain for bounded potentials, assuming that the Neumann data measured on slightly more than the half of the boundary. In \cite{kenig2007calderon}, Kenig, Sj\"ostrand, and Uhlmann proved uniqueness for the Schr\"odinger operator under the assumption that the Dirichlet-to-Neumann map measured on an open subset of the boundary $\partial\Omega$ where the Dirichlet data is supported on a complementary part of the boundary. In \cite{ferreira2007determining, chung2014partial,Bh18}, authors investigated partial data inverse problems for the magnetic Schr\"odinger operator $(D+A)^2+q$ in a bounded simply-connected domain for some vector field $A$ and potential $q$. In \cite{krupchyk2012determining,krupchyk2014inverse} Krupchyk Lassas and Uhlmann introduced Calde\'on type inverse problems for biharmonic and polyharmonic operators. They were soon followed by \cite{Yang-Paper,ghosh2016determination,BG_19,BhattacharyyaGhosh20,bhattacharyya2021unique,SS_linearize_polyharmonic}, etc., where the authors have investigated uniqueness for inverse boundary value problems for biharmonic and polyharmonic operators from full and partial boundary data. So far, the question of recovering coefficients of a higher order elliptic operator from partial boundary data where a part of the boundary is inaccessible has been open in the literature.
In \cite{Yang-Paper}, authors considered domains with inaccessible part of the boundary probed with a biharmonic operator with just first and a zeroth order perturbation. On the other hand, in \cite{bhattacharyya2021unique}, the authors have considered polyharmonic operators with several lower order tensorial perturbations while the entire boundary is accessible.
In this article, we generalise the existing results by considering several tensorial perturbations as in \eqref{Operator} and an inaccessible part of the boundary $\Gamma_{0}=\partial \Omega\setminus\Gamma$, assuming that $\Gamma_0$ is part of a hyperplane.

In dimension $n=3$, Isakov \cite{IsakovPartialData07} introduced the reflection approach for the inverse boundary value problems for the Schr\"odinger equation and the conductivity equation for dealing with inverse problems in domains with inaccessible part of the boundary. 
Kenig and Salo \cite{KS_Survey} unified and improved the approaches of \cite{IsakovPartialData07} and \cite{kenig2007calderon} for the Schr\"odinger operator. In \cite{KLU12}, the authors considered the magnetic Schr\"odinger operator $(-\Delta)+A\cdot D+q$ and proved that the magnetic field $dA$ and potential $q$ can be uniquely recovered from the partial Dirichlet-to-Neumann map. In \cite{Yang-Paper}, the author proved recovery of first and zeroth order perturbations for inverse boundary value problems for the biharmonic operator on a domain from the partial Dirichlet-to-Neumann map. Some related works can also be found in \cite{caro2009inverse, Caro2011inverse, LiU10,liu2021determine}. 
To the best of the author's knowledge, this is the first injectivity result for polyharmonic operators with higher order tensorial perturbations, in a domain having an inaccessible part of the boundary.

The novelty of this article lies in recovering tensorial perturbations having a nontrivial inaccessible part of the boundary. We assume the inaccessible part of the boundary is part of a hyperplane.
Here we briefly sketch the idea of the proof.
In order to deal with the inaccessible part of the boundary, we first extend the domain and the operator \eqref{Operator} by considering a suitable reflection about the hyperplane where the inaccessible part is supported. For $m=1,2$, similar types of reflections have been considered in \cite{IsakovPartialData07,Yang-Paper,Leyter-localdata-17}. In the extended domain, we construct a class of Complex Geometric Optics type solutions using appropriate Carleman estimates. The reflection and the construction of solutions are done in such a way that the solution of the extended operator satisfies the original boundary value problem in $\Omega$ upon folding the domain about the inaccessible part of the boundary. Thus we obtain a class of solutions for the boundary value problem \eqref{Direct Problem}. Using these solutions, a careful analysis entails having a class of integral equations of the lower order tensorial perturbations. Upon further analysis of these solutions, we reduce our problem to an uniqueness question for a series of Momentum Ray Transforms (MRT) of the unknown tensorial perturbations. Recently in \cite{bhattacharyya2021unique,SS_linearize_polyharmonic} MRT has been proved to be useful for recovering higher order tensors from integral equations. We use uniqueness results of MRT from \cite{bhattacharyya2021unique, SS_linearize_polyharmonic} to finally obtain uniqueness of the tensorial perturbations and thus complete the proof of Theorem \ref{Main_Result}.

This paper is organized as follows. In Section \ref{Sec_Preliminaries}, we provide the necessary preliminary results for proving our main results. These include interior Carleman estimates, constructions of Complex Geometric Optics (CGO) solutions, and momentum ray transforms. 
In Section \ref{Sec_Proof_Thm}, we prove our main theorem by constructing special CGO solutions and deriving appropriate integral identities to recover the coefficients.

\section{Preliminary results}\label{Sec_Preliminaries}
In this section, we start by recollecting the construction of the Complex Geometric Optics (CGO) solutions for the operator \eqref{Operator} in a bounded domain. We use Carleman estimates in the construction. We refer to \cite{krupchyk2012determining,krupchyk2014inverse,ghosh2016determination,bhattacharyya2021unique, SS_linearize_polyharmonic} for a detailed description of the construction.

\subsection{Interior Carleman estimates}
Following \cite{kenig2007calderon}, we start by recalling the definition of limiting Carleman weights for the semiclassical Laplacian $(-h^{2} \Delta)$, where $h>0$ is a small parameter. Let $\widetilde{\Omega}$ be an open set in $\mathbb{R}^{n}$ such that $\Omega \subset \subset \widetilde{\Omega}$ and $\varphi \in C^{\infty}(\widetilde{\Omega}, \mathbb{R})$ with $\nabla \varphi \neq 0$ in $\overline{\Omega}$. We consider the conjugated semiclassical Laplacian operator 
\begin{equation}\label{conjugated operator}
P_{0, \varphi} := e^{\frac{\varphi}{h}}\left(-h^{2} \Delta\right) e^{-\frac{\varphi}{h}}, \quad 0<h\ll 1,
\end{equation}
and its semiclassical principal symbol
$ {p}_{0, \varphi}(x, \xi) = |\xi|^{2} + 2\mathrm{i}\nabla \varphi\cdot \xi - |\nabla\varphi|^{2}$.
\begin{definition}[\cite{kenig2007calderon}]\label{LCW}
We say that $\varphi \in C^{\infty}(\widetilde{\Omega})$ is a limiting Carleman weight for $P_{0,\varphi}$ in $\Omega$, if $\nabla \varphi \neq 0$ in $\overline{\Omega}$ and $\operatorname{Re}\left(p_{0, \varphi}\right), \operatorname{Im}\left(p_{0, \varphi}\right)$ satisfies
\begin{equation*}
 \Big{\{}\mathrm{Re}(p_{0,\varphi}), \mathrm{Im}(p_{0,\varphi})\Big{\}}(x,\xi)=0 \mbox{ whenever } p_{0,\varphi}(x,\xi)=0 \mbox{ for } (x,\xi)\in \widetilde{\Omega}\times(\mathbb{R}^n\setminus\{0\}),
\end{equation*}
where $\{\cdot, \cdot \}$ denotes the Poisson bracket.
\end{definition}
Examples of such $\varphi$ are linear weights $\varphi(x)=\mu_{1}\cdot x$, where $0 \neq\mu_{1}\in \mathbb{R}^{n}$ or logarithmic weights $\varphi(x)=\log \left|x-x_{0}\right|$ with $x_{0} \notin \overline{\widetilde{\Omega}}.$

We shall first introduce semiclassical Sobolev spaces. For any $h>0$, be a small parameter and an open set $\Omega \subset\mathbb{R}^n$ and for any non-negative integers $m$, the semiclassical Sobolev space $H^m_{\mathrm{scl}}(\Omega)$ is the space $H^m(\Omega)$ endowed with the following semiclassical norm
\[
 \lVert u \rVert^2_{H^m_{\mathrm{scl}}(\Omega)} = \sum\limits_{|\alpha|\le m} \lVert ( hD)^{\alpha}\, u \rVert^2_{L^2(\Omega)}.
\]
For $\Omega =\R^{n}$ we can equivalently define the semiclassical Sobolev spaces $H_{scl}^{s}(\mathbb{R}^{n})$, $s\in \mathbb{R}$ as the space $H^{s}(\mathbb{R}^{n})$ endowed with the semiclassical norm 
\[
\|u\|_{H^s_{\mathrm{scl}}\left(\mathbb{R}^{n}\right)}^{2}=\left\|\langle h D\rangle^{s} u\right\|_{L^{2}\left(\mathbb{R}^{n}\right)}^{2}= \int_{\R^{n}}(1+h^{2}|\xi|^{2})^{s}|\hat{u}(\xi)|^{2}\mathrm{d}\xi,
\]
here $\langle \xi \rangle = (1+|\xi|^{2})^{\frac{1}{2}}$. For a detailed description of semiclassical analysis and Poisson bracket present here, see \cite{Martinez,Zworski}.
We recall the following Carleman estimate from \cite{bhattacharyya2021unique}.
\begin{proposition}\label{Carleman_Poly}
Let the coefficients $A^{j}$ belong to $C_{c}^{\infty}(\Omega; \mathbb{C}^{n^{j}})$ for $1 \leq j \leq m$ and $q \in L^{\infty}(\Omega; \mathbb{C})$. Let $\varphi(x)$ be a limiting Carleman weight for the conjugated semiclassical Laplacian. Then for $0<h \ll 1$ and $ -2m \leq s\leq 0$, we have
    \begin{equation*}
    \label{Carleman estimate}
h^{m}\|u\|_{H^{s+2m}_{\mathrm{scl}}(\Omega)} \leq C\left\|h^{2 m} e^{\frac{\varphi}{h}} \mathcal{L}(x, D) e^{-\frac{\varphi}{h}}u\right\|_{H^{s}_{\mathrm{scl}}(\Omega)}, \text { for all } u \in C^{\infty}_{c}(\Omega),
\end{equation*}
where the constant $C=C_{s, \Omega, A^{j}, q}$ is independent of $h$.
\end{proposition}
Let us denote
\[
\mathcal{L}_{\varphi}(x, D)=h^{2 m} e^{\frac{\varphi}{h}} \mathcal{L}(x, D) e^{-\frac{\varphi}{h}}.
\]
Let $\mathcal{L}_{\varphi}^{*}(x, D)$ be the formal $L^2$ adjoint of $\mathcal{L}_{\varphi}(x, D)$ given by $(\mathcal{L}_{\varphi}u, v)_{L^{2}(\Omega)} = (u, \mathcal{L}_{\varphi}^{*}v)_{L^{2}(\Omega)}$ for all $u, v \in C^{\infty}_{c}(\Omega)$.
It follows from a straightforward computation using integration by parts  that $\mathcal{L}^{*}(x, D)$ has a similar form as 
$\mathcal{L}(x, D)$ with the same regularity of the coefficients. Since $-\varphi$ is also a limiting Carleman 
weight if $\varphi$ is, the Carleman estimate derived in Proposition \ref{Carleman_Poly} holds true for $\mathcal{L}_{\varphi}^{*}(x, D)$ as well. We convert the Carleman estimate for $\mathcal{L}_{\varphi}^{*}(x, D)$ into a solvability result for $ \mathcal{L}_{\varphi}(x, D)$.
\begin{proposition}\label{Soblvability_result}
Let the coefficients $A^{j}$ for $1 \leq j \leq m$ belong to $C_{c}^{\infty}(\Omega, 
\mathbb{C}^{n^{j}})$ and $q\in L^{\infty}(\Omega, \mathbb{C})$, and let $\varphi(x)$ be a limiting Carleman weight the conjugated semiclassical Laplacian. Then for any $v\in L^{2}(\Omega)$ and small enough $h>0$ there exists $u \in H^{2m}(\Omega)$ such that 
\begin{equation}
    \mathcal{L}_{\varphi}(x, D) u=v \text { in } \Omega,
\end{equation}
satisfying, 
\[
h^{m}\|u\|_{H_{\mathrm{scl}}^{2m}} \leq C\|v\|_{L^{2}(\Omega)},
\]
where $C>0$ is independent of $h$ and depends only on $A^{j}, \text{ for } 1\leq j \leq m,$ and $q$.
\end{proposition} 
The proof of the above Proposition follows from Hahn--Banach extension theorem and Riesz representation theorem (see \cite{ferreira2007determining,krupchyk2012determining,krupchyk2014inverse}).

\subsection{Construction of Complex Geometric Optics solutions}\label{CGOs_ON_BOUNDED_DOAMIN}
We construct complex geometric optics solutions of the equation $\mathcal{L}(x, D)u = 0$ in $\Omega$, based on Proposition \ref{Soblvability_result}. Following \cite{ghosh2016determination,bhattacharyya2021unique}, we propose the solution to be of the form:
\begin{equation}\label{Propose CGO}
u(x)=u(x ; h) =e^{\frac{\varphi+\mathrm{i} \psi}{h}}\left(a_{0}(x)+h a_{1}(x)+h^{2} a_{2}(x)\cdots +h^{m-1}a_{m-1}(x)+r(x ; h)\right):= e^{\frac{\varphi+\mathrm{i}\psi}{h}}(\mathrm{a}(x;h)),
\end{equation}
where $h>0$ is a small parameter, $\varphi(x)$ is a limiting Carleman weight for the semiclassical Laplacian, the real-valued phase function $\psi$ is chosen such that $\psi$ is smooth near $\overline{\Omega}$ and solves the Eikonal equation $p_{0, \varphi}(x, \nabla \psi)=0$ in $\widetilde{\Omega}$.
The functions $\{a_{j}(x)\}$, for  $0\leq j\leq m-1$, are the complex amplitudes solving certain transport equations and $r(x;h)$ is the correction term which is determined later. We consider $\varphi$ and $\psi$ of the form 
\begin{equation}\label{Carleman_Weight}
\varphi(x)= \mu^{(2)}\cdot x \quad \text{ and }\quad \psi(x) = \left(\frac{h\xi}{2}+\sqrt{1-h^2\frac{|\xi|^2}{4}}\mu^{(1)}\right)\cdot x,
\end{equation}
where $\xi,\mu^{(1)},\mu^{(2)}\in\R^n$ be such that $|\mu^{(1)}|=|\mu^{(2)}|=1$ and $\mu^{(1)}\cdot\mu^{(2)}=\mu^{(1)}\cdot\xi=\mu^{(2)}\cdot\xi=0$.
One can easily check that $\varphi$ and $\psi$ solves the Eikonal equation $p_{0,\varphi}(x, \nabla \psi)$ in $\widetilde{\Omega}$, that is $|\nabla \varphi|=|\nabla \psi|$ and $\nabla \varphi \cdot \nabla \psi=0$.
We prove that there exist $h_0>0$ such that $u(x)$ is a solution of \eqref{Operator} for each fixed $h\in(0,h_0)$. We now prove the following lemma.
    \begin{lemma}\label{Lemma}
        Consider the equation
        \begin{equation}\label{Poly_Eqn}
            \mathcal{L}(x, D)u= \Big\{(-\Delta)^m + \sum_{j=0}^{m-1} \sum_{i_{1}, \cdots, i_{m-j}=1}^{n} A_{i_{1} \cdots i_{m-j}}^{m-j}(x) D_{i_{1} \cdots  
 i_{m-j}}^{m-j}+q(x)\Big\}u=0,
        \end{equation}
        where $A^{j}\in C_{c}^{\infty}(\Omega, \mathbb{C}^{n^{j}}) \text{ for } 1 \leq j \leq m$ and $q(x) \in L^{\infty}(\Omega, \mathbb{C})$. We assume that for $j=m,$ $A^{j}$ is partially isotropic. Then for $h>0$ small enough, there exists a solution $u\in H^{2m}(\Omega)$ of \eqref{Poly_Eqn} of the form
        \begin{equation}
\begin{aligned}
u(x ; h) = &e^{\frac{\varphi+\mathrm{i} \psi}{h}}\left(a_{0}(x)+h a_{1}(x)+h^{2} a_{2}(x)\dots+ h^{m-1}a_{m-1}(x)+r(x ; h)\right):= e^{\frac{\varphi+\mathrm{i}\psi}{h}}(\mathrm{a}(x;h)), 
\end{aligned}
\end{equation}
where  $\varphi$ and $\psi$ are defined as in equation \eqref{Carleman_Weight}, and the functions $a_{0}, \cdots , a_{m-1}$ satisfy the transport equations \eqref{Transport for a_{0}} and \eqref{System of transport}, and $r$ satisfies the estimate
\[
\|r(x;h)\|_{H_{\mathrm{scl}}^{2 m}(\Omega)} \leq C h^{m},
\] where $C$ is independent of $h$.
    \end{lemma}
        
    \begin{proof}
        Consider the conjugated operator 
        \begin{align}\notag\label{Conjugated operator}
e^{-\frac{\varphi+i \psi}{h}} \mathcal{L}(x, D) e^{\frac{\varphi+i \psi}{h}} \mathcal{A}(x ; h)=&\left(-\frac{1}{h} T-\Delta\right)^{m} \mathcal{A}(x ; h)
-\sum_{i=1}^{n} \overline{A}_{i_{1}\ldots i_{m-2} }^{m-2}\mathcal{I}_{i_{1} \ldots i_{m-2}}^{m-2}\left(\frac{1}{h} T+\Delta \right) \mathcal{A}(x ; h)\\
&+\sum_{j=1}^{m-1} \sum_{i_{1}, \cdots, i_{j}=1}^{n} A_{i_{1} \ldots i_{j}}^{j} \mathcal{I}_{i_{1} \ldots i_{j}}^{j}\mathcal{A}(x ; h)+q(x) \mathcal{A}(x ; h),
\end{align}
where 
\[
 \mathcal{I}_{i_{1} \ldots i_{j}}^{j}=e^{-\frac{\varphi+\mathrm{i} \psi}{h}} D_{i_{1} \ldots i_{j}}^{j} e^{\frac{\varphi+\mathrm{i} \psi}{h}}=\prod_{k=1}^{j}\left(\frac{1}{h} D_{i_{k}}(\varphi+\mathrm{i} \psi)+D_{i_{k}}\right),
\] 
\[
T = 2 \nabla(\varphi+\mathrm{i} \psi) \cdot \nabla + \Delta(\varphi+\mathrm{i} \psi)\quad \text { and } \quad \mathcal{A}(x ; h):= \sum_{j=0}^{m-1}h^{j}a_{j}(x).
\]
Now substituting \eqref{Propose CGO} in \eqref{Poly_Eqn}, we get the following transport equations for the amplitudes $a_{j}(x)$ with $0\leq j \leq m-1$ by equating the coefficient of $h^{-m+j}$ to zero for $0\leq j \leq m-1$ as
\begin{align}\label{Transport for a_{0}}
        T^{m} a_{0}(x)&=0, \quad \text { in } \quad \Omega.
        \end{align}
        \begin{align}\label{System of transport}
      T^{m} a_{j}(x)=-\sum_{k=1}^{j} M_{k} a_{j-k}(x), \quad \text { in } \quad \Omega, \quad 1 \leq j \leq m-1 .
    \end{align}

From \eqref{Conjugated operator}, we can calculate the differential operators $M_{j}$ of order $m+j$ where $1\leq j \leq m-1$. It is well-known that the solutions of \eqref{System of transport} are smooth; see \cite{ferreira2007determining}. We discuss an explicit form for the smooth solution $a_{0}(x)$ solving \eqref{Transport for a_{0}} in the next section.  Using $a_{0}(x), \cdots, a_{m-1}(x) \in C^{\infty}(\Omega)$ satisfying \eqref{System of transport}, \eqref{Transport for a_{0}}, we see that 
$$
e^{-\frac{\varphi+\mathrm{i} \psi}{h}} \mathcal{L}(x, D) e^{\frac{\varphi+\mathrm{i} \psi}{h}} \mathcal{A}(x ; h)\simeq \mathcal{O}(1).
$$
Now if $u(x ; h)$ as in \eqref{Propose CGO} is a solution of $\mathcal{L}(x, D) u(x ; h)=0$ in $\Omega$, we see that
\begin{equation}
0=e^{-\frac{\varphi+\mathrm{i} \psi}{h}} \mathcal{L}(x, D) u=e^{-\frac{\varphi+\mathrm{i} \psi}{h}} \mathcal{L}(x ; D) e^{\frac{\varphi+\mathrm{i} \psi}{h}}(\mathcal{A}(x ; h)+r(x ; h)) .
\end{equation}
This implies
$$
e^{-\frac{\varphi+\mathrm{i} \psi}{h}} \mathcal{L}(x, D) e^{\frac{\varphi+\mathrm{i} \psi}{h}} r(x ; h)=F(x ; h), \quad \text { for some } F(x ; h) \in L^{2}(\Omega), \quad \text { for all } h>0 \text { small. }
$$
By our choices, $a_{j}(x)$ annihilates all the terms of order $h^{-m+j}$ in $e^{-\frac{\varphi+\mathrm{i} \psi}{h}} \mathcal{L}(x ; D) e^{\frac{\varphi+\mathrm{i} \psi}{h}} \mathcal{A}(x ; h)$ in $\Omega$ for $j=0, \ldots, m-1$. Thus we get $\|F(x, h)\|_{L^{2}(\Omega)} \leq C$, where $C>0$ is uniform in $h$ for $h \ll 1$.
Using Proposition \ref{Soblvability_result} we have the existence of $r(x ; h) \in H^{2 m}(\Omega)$ solving
$$
e^{-\frac{\varphi+\mathrm{i} \psi}{h}} \mathcal{L}(x, D) e^{\frac{\varphi+\mathrm{i} \psi}{h}} r(x ; h)=F(x ; h),
$$
with the estimate
$$
\|r(x ; h)\|_{H_{\mathrm{scl}}^{2 m}(\Omega)} \leq C h^{m}, \quad \text { for } h>0 \text { small enough. }
$$
\end{proof}
Similarly, we can construct the CGO solution of the adjoint equation $\widetilde{\mathcal{L}}^{*}(x, D)v(x;h) = 0$ in $\Omega$.

\subsection{Momentum ray transforms}\label{Subsec_MRT}
    Here we briefly introduce the ray transform and Momentum Ray Transforms (MRT) of symmetric tensor fields and describe some injectivity results of these transforms. MRT was introduced in \cite{ Sharafutdinov_book} and later studied by \cite{Krishnan2018,abhishek2019support,Krishnan2019a,  SumanSIAM,agrawal2022unique,Mishra_Suman_2023,bhattacharyya2021unique,ilmavirta2023unique,jathar2024normal}. 

Recall that $S^{m}:=S^{m}(\mathbb{R}^{n})$ is the space of symmetric $m$-tensor fields and $C^{\infty}(\R^{n};S^{m})$ denotes the space of symmetric $m$-tensor fields in $\R^{n}$ whose components are smooth.
We write $C_{c}^{\infty}(\R^{n};S^{m})$ to denote tensors in $C^{\infty}(\R^{n};S^{m})$ whose components are compactly supported. Throughout the section, we assume the Einstein summation convention for repeated indices.
\begin{definition}[\cite{Sharafutdinov_book}, Chapter 2]\label{Ray_trans}
The ray transform is the bounded linear operator 
  \begin{equation}
      I:C_{c}^{\infty}(\R^{n};S^{m})\to C^{\infty}(\R^{n}\times \R^{n}\setminus \{0\})
      \end{equation}
      defined as 
      \begin{equation}
If(x,\xi) = \int_{-\infty}^{\infty}  f_{i_1\cdots i_m}(x+t\xi)\, \xi^{i_1}\,\cdots \xi^{i_m}\, \mathrm{d}t = \int_{-\infty}^{\infty}\langle {f(x+t\xi), \xi^{m}}\rangle \mathrm{d}t,
\end{equation}
\end{definition}
for all $(x,\xi)\in \mathbb{R}^n\times \mathbb{R}^n \setminus\{0\}$ which determines the lines $\{x+t\xi \mid t\in \R \}$. Hereafter $\langle \cdot, \cdot \rangle$ is the standard dot-product on $\R^{n}$.
The Kernel description of the ray transform is given by the following Saint Venant operator.
\begin{definition}[\cite{Sharafutdinov_book}, Chapter 2]
The Saint Venant operator 
   \[
   W : C^{\infty}(\R^{n};S^{m}) \to C^{\infty}(\R^{n};S^{m}\otimes S^{m})\] 
   is defined by
   \[
(Wu)_{i_{1}\cdots i_{m}j_{1}\dots j_{m}} = \sigma(i_{1}\cdots i_{m})\sigma({j_{1}\dots j_{m}})\sum_{p=0}^{m}(-1)^{p}\binom mp  \frac{\partial^{m} u_{i_{1}\dots i_{m-p}j_{1}\cdots j_{p}}}{\partial x_{j_{p+1}}\cdots \partial x_{j_{m}}\partial x_{i_{m-p+1}}\cdots \partial x_{i_{m}}},
   \]
   where $\binom mp  
   =\frac{m!}{p!(m-p)!}$ are binomial coefficients, $\otimes$ denotes the product of tensors, and $\sigma$ is the symmetrization operator.
   \end{definition}
Let us recall the injectivity result of the ray transform of symmetric tensor fields.
\begin{theorem}[\cite{Sharafutdinov_book}, Theorem 2.2.1]\label{Inj_Ray_Trans}
    Let $n\geq 2$ and $m$ be non-negative integers, $l= \text { max }\{m,2\}$. For a compactly supported field $f\in C^{l}(\R^{n};S^{m})$ the following statements are equivalent:
        \begin{enumerate}
            \item $If = 0;$
        \item  there exist a compactly-supported field $v\in C^{l+1}(\R^{n}; S^{m-1})$ such that its support is contained in the convex hull of the support of $f$ and 
        \[dv = f;\]  
        \item  the equality $Wf = 0$ is valid in $\R^{n}$.
        \end{enumerate}
\end{theorem}
From Theorem \ref{Inj_Ray_Trans}, we see that the operator $I$ has an infinite dimensional kernel for $m>0$. A symmetric tensor field can be uniquely decomposed into a solenoidal and a potential part in a bounded domain \cite{Sharafutdinov_book}, and the potential part lies in the kernel of the ray transform. Therefore, unique recovery of entire tensor fields $f$ is impossible from $If$ requires some additional information, which leads to the study of the MRTs. For each integer $k\geq0$, Momentum Ray Transforms are bounded linear operators
\[I^{k}:C_{c}^{\infty}(\R^{n};S^{m})\to C^{\infty}(\R^{n}\times \R^{n}\setminus \{0\}),\]  defined as 
\begin{equation*}
I^kf(x,\xi) := \int_{\R} t^k \, f_{i_1\cdots i_m}(x+t\xi)\, \xi^{i_1}\,\cdots \xi^{i_m}\, \mathrm{d}t \quad \text { for all } \quad (x,\xi)\in \mathbb{R}^n\times \mathbb{R}^n \setminus\{0\}.
\end{equation*}
Notice that, for $k=0, I^{0}f= If$ is the classical ray transform \eqref{Ray_trans}. 
It has been established that \cite{Sharafutdinov_book,Krishnan2018} the knowledge of the MRTs $I^k$ for $k=0,\cdots,m$ uniquely recovers a symmetric $m$-tensor, whereas having any less information would lead to nontrivial kernels.
For the purpose of our result, we define the MRT on the unit sphere bundle. For $f\in C_{c}^{\infty}(\R^{n}; S^{m})$ and each integer $k\geq 0$, we define the restricted MRT $J^{k}$ as
\[
J^kf(x,\xi) := \int_{\R} t^k \, f_{i_1\cdots i_m}(x+t\xi)\, \xi^{i_1}\,\cdots \xi^{i_m}\, \mathrm{d}t = I^{k}f(x,\xi)|_{\R^{n}\times \mathbb{S}^{n-1}}, \quad \mbox{for all } (x,\xi)\in \R^{n}\times \mathbb{S}^{n-1}.
\]
In \cite{Sharafutdinov_book,Krishnan2018,abhishek2019support}, the authors showed the rank $m$ symmetric tensor field $f$ can be uniquely determined by the function $\mathcal{J}^{m}f:= (J^{0}f, J^{1}f,\dots, J^{m}f)$.

In our setup, we end up with a modified version of the above restricted MRTs, where we deal with a sum consisting of MRTs of a function, a vector field, up to a symmetric $m$-tensor field. Moreover, the setup is even more complicated due to the presence of the complex vectors in place of the real vectors $\xi \in \R^{n}$. 
In \cite{bhattacharyya2021unique}, the authors define MRT of $F\in C_{c}^{\infty}(\R^{n}; \mathbf{S}^m\R^{n})$, where
$\mathbf{S}^m=S^0\oplus S^1\oplus\cdots\oplus S^m.$
Any such $F$ can be described uniquely as
\begin{equation}
\begin{aligned}
F &=\sum_{p=0}^{m}f^{(p)}=f^{(0)}_{i_0}+f^{(1)}_{i_1}\,\D x^{i_1} + \cdots+f^{(p)}_{i_1\cdots i_p}\D x^{i_1}\cdots \D x^{i_p}+\cdots+f^{(m)}_{i_1\cdots i_m}\D x^{i_1}\cdots \D x^{i_m}\\
    & \hspace{1cm}	\quad=\left( f^{(0)}_{i_0},f^{(1)}_{i_1},\cdots, f^{(m)}_{i_1\cdots i_m} \right),
\end{aligned}
\end{equation}
which can be viewed as the sum of a function, a vector field, and up to a symmetric $m$-tensor field with $f^{(p)}\in S^{p}$ for each $0\leq p \leq m$.
Next, we recall the definition of MRT of $F\in C_{c}^{\infty}(\R^{n})$ from \cite{bhattacharyya2021unique}.
For every integer $k\geq 0$ and for all $(x,\xi)\in \mathbb{R}^n\times \mathbb{R}^n \setminus\{0\}$ we define
\begin{equation}\label{MRT of F}
    \begin{aligned}
        I^{m,k}F(x,\xi)&= \sum_{p=0}^{m} I^{k} f^{(p)}(x, \xi)\\
        &= \int_{\R} t^k \left(f^{(0)}_{i_0}(x+t\xi)+ f^{(1)}_{i_1}(x+t\xi)\, \xi^{i_1}+\cdots+ f^{(m)}_{i_1\cdots i_m}(x+t\xi)\, \xi^{i_1}\,\cdots \xi^{i_m}\right)\mathrm{d}t.
        \end{aligned}
\end{equation}
We conclude this section by stating the following injectivity result for the MRT, which we will apply repeatedly in the proof of Theorem \ref{Main_Result}. For the proof of the following result, \cite[Lemma 3.7]{bhattacharyya2021unique}.
\begin{theorem}\label{Lemma_MRT}
    Let $n\geq 3$ and $f$ be a smooth compactly supported symmetric $m$-tensor field in $\mathbb{R}^{n}$.  Suppose that for all unit vectors $\eta\perp e_{1}$, we have that
\begin{equation}\label{Injectivity_MRT}
\sum_{i_{1},\ldots,i_{m} =1}^{n}\int_{\mathbb{R}} x_{2}^{k} f_{i_{1} \cdots i_{m}}(0, x_{2}, x^{\prime\prime})\left(\prod_{j=1}^{m}\left(e_{1}+\mathrm{i}\eta\right)_{i_{j}}\right) \mathrm{d} x_{2}=0, \quad \text { for a.e. } x^{\prime\prime} \in \mathbb{R}^{n-2}, 
\end{equation}
\text{ for each } $0\leq k \leq m.$
 Then, 
$$
f(0, x^{\prime})=i_{\delta} v(0, x^{\prime}), \quad \text{ for } \quad m \geq 2,
$$
where $e_{1}$ is the standard basis vector in $\R^{n}$ and  $v$ is a symmetric $(m-2)$-tensor field compactly supported in $x^{\prime}$ variable and
$$
f(0, x^{\prime})=0 \quad \text { for } \quad m=0,1.
$$
\end{theorem}
\begin{remark}
One can rewrite the identity \eqref{Injectivity_MRT} as 
 \[\int_{\R} x_{2}^{k} F(0,x_{2},x^{\prime\prime}) \mathrm{d}x_{2} = 0 \quad \text { for } \quad 0\leq k \leq m,\]
where $F$ denotes the sum of symmetric tensors in $i_{j}= 2, \cdots , n$ indices for $1\leq j \leq m$ as
\[
F(0,x_{2},x^{\prime\prime})= \sum_{p=0}^{m} \mathrm{i}^{p}\tilde{f}^{p}_{i_{1} \cdots i_{p}}(0, x_{2}, x^{\prime\prime})(\eta)_{i_{1}} \cdots (\eta)_{i_{p}}
\quad \text { with }\quad  \tilde{f}^{p}_{i_{1} \cdots i_{p}}= \binom{m}{p} f_{i_{1} \cdots i_{p} 1 \cdots 1}.
\]

Also, note that the presence of non-linearity for the tensor $m\geq2$ generates a nontrivial kernel $v$ in Theorem \ref{Lemma_MRT}.
\end{remark}

\section{Proof of the main result}\label{Sec_Proof_Thm}
In this section, we prove the Theorem \ref{Main_Result}. We divide this section into three parts. In the first part we construct special solutions to the \eqref{Operator} with Dirichlet data vanishing on $\Gamma_0$. 
Next, we use the boundary data to obtain a series of momentum ray transforms of the unknown coefficients. Finally, we determine the parameters using the injectivity in Theorem \ref{Lemma_MRT} repeatedly with particular choices of solutions.

\subsection{Construction of special solutions}\label{Subsec_special_Sol}
Our main goal of this section is to construct complex geometric optics solutions $U\in H^{2m}(\Omega)$ for the polyharmonic equation $\mathcal{L}(x, D)U = 0 \text{ in } \Omega$ with Dirichlet boundary conditions $(-\Delta)^{k-1}U|_{\Gamma_{0}} = 0 \text{ for } 1\leq k \leq m$. We use a reflection argument as in \cite{IsakovPartialData07} to construct these solutions.
We reflect $\Omega$ with respect to the plane $x_{n} = 0$ and denote this reflection by 
\[\Omega^{*} = \{ (x^{\prime}, -x_{n}) :  x =(x^{\prime},x_{n})\in \Omega \}\quad \text { where }\quad  x^{\prime} = \left(x_{1}, \cdots , x_{n-1}\right)\in \mathbb{R}^{n-1}.
 \]
We denote $x^{*} = \left(x^{\prime}, -x_{n}\right)\in \Omega^{*}$ for any $x = \left(x^{\prime}, x_{n}\right)\in \Omega, \text { and } f^{*}(x) = f(x^{*})$ for any function $f$ on $\Omega$. 
We set
 \[
 O= \Omega \cup\text{int($\Gamma_{0}$) } \cup \Omega^{*},
\]
 where int$(\Gamma_{0})\subset \{x_{n} = 0\}$.
We extend the coefficients $q$ and $A^{j}$ for $1\leq j \leq m$ from $\Omega$ to $O$ as
\begin{equation}\label{reflect}
\begin{aligned}
q(x)&:= 
\begin{cases}
q(x) \quad &\text{if } x\in \Omega, \\
q(x^{*}) &\text{if } x\in \Omega^{*},
\end{cases}\\
\left(A^{j}\right)_{\underbrace{n\cdots n}_{p-\text{times}} i_{1}\cdots i_{j-p} }(x) &:=
\begin{cases}
\left(A^{j}\right)_{nn\cdots n i_{1}\cdots i_{j-p} }(x)  \quad & \text{if } x\in \Omega, \\
(-1)^{p}\left(A^{j}\right)_{nn\cdots n i_{1}\cdots i_{j-p} }(x^{*})   &\text{if } x\in \Omega^{*},
\end{cases}
\end{aligned}
\end{equation}
where $ 1 \leq i_1 ,\dots,i_{j-p} \leq n-1,$ $ p$ := number of indices are equal to $n$. Note that, $A^{j},q\in C_{c}^{\infty}(O, \mathbb{C}^{n^{j}})$, for $1\leq j \leq m$.

\begin{proposition}\label{Existence of special solutions}
Let $A^{j}\in C_c^{\infty}({\Omega}, \mathbb{C}^{n^{j}})$ be symmetric for $1\leq j \leq m$ with $A^m$ to be partially isotropic, $q\in L^{\infty}\left(\Omega, \mathbb{C}\right)$ and $h>0$ be small enough. 
Then there exist a solution $U\in H^{2m}\left(\Omega\right)$ to the equation $\mathcal{L}(x, D)U = 0$ in $\Omega$ with $(-\Delta)^{k-1}U|_{\Gamma_{0}} = 0 \text{ for } 1\leq k \leq m $
of the form 
\begin{equation}
     U(x;h) = \mathcal{U}(x; h) - \mathcal{U}(x^{*}; h), \quad x\in \Omega,
\end{equation}
where $ \mathcal{U}\in H^{2m}(O)$
solving $\mathcal{L}(x, D) \mathcal{U}(x ; h)=0$ in $O$ and has the form
\begin{equation}
\begin{aligned}
\mathcal{U}(x ; h) = &e^{\frac{\varphi+\mathrm{i} \psi}{h}}\left(a_{0}(x)+h a_{1}(x)+h^{2} a_{2}(x)\dots+ h^{m-1}a_{m-1}(x)+r(x ; h)\right), \quad x\in O.
\end{aligned}
\end{equation}
Here  $\varphi$ and $\psi$ are defined as in \eqref{Carleman_Weight}, and the functions $a_{0}, \cdots , a_{m-1}$ satisfy the transport equations \eqref{Transport for a_{0}} and \eqref{System of transport} in $O$, and $r$ satisfies the estimate
\begin{equation}\label{Decay r}
\|r(x;h)\|_{H_{\mathrm{scl}}^{2 m}(O)} \leq C h^{m},
\end{equation}
where $C$ is independent of $h$.
\end{proposition}
\begin{proof} 
We extend the operator $\mathcal{L}(x, D)$ over $O$ by extending $q$ and $A^{j}$ for $1\leq j\leq m$ as \eqref{reflect}. Applying Lemma \ref{Lemma} to $A^{j}\in C_{c}^{\infty}(O)$ and $q\in L^{\infty}(O)$ we get a CGO solution $\mathcal{U}(x; h)\in H^{2m}(O)$ of $\mathcal{L}(x, D) \mathcal{U}(x; h)=0$ in $O$ as
\begin{equation}
\mathcal{U}(x; h) = e^{\frac{\varphi+\mathrm{i} \psi}{h}}\left(a_{0}(x)+h a_{1}(x)+h^{2} a_{2}(x)\dots+ h^{m-1}a_{m-1}(x)+r(x ; h)\right), \text{ in } O, 
\end{equation}
with the decay estimate \eqref{Decay r}.
A direct calculation shows
\[\mathcal{L}(x, D)U(x; h)= \mathcal{L}(x, D)\left(\mathcal{U}(x; h)-\mathcal{U}(x^{*}; h)\right) =0, \quad x\in \Omega\subset \left\{\mathbb{R}^{n};
 x_{n}>0\right\}, \]
and $(-\Delta)^{k-1}U|_{\Gamma_{0}}=0$ for all  $k=1,2, \cdots m$.
\end{proof}

We recall our choices of phase functions
\begin{equation*}
\varphi(x)= \mu^{(2)}\cdot x \quad \text{ and }\quad \psi(x) = \left(\frac{h\xi}{2}+\sqrt{1-h^2\frac{|\xi|^2}{4}}\mu^{(1)}\right)\cdot x ,
\end{equation*}
where $\xi,\mu^{(1)},\mu^{(2)}\in\R^n$ be such that $|\mu^{(1)}|=|\mu^{(2)}|=1$ and $\mu^{(1)}\cdot\mu^{(2)}=\mu^{(1)}\cdot\xi=\mu^{(2)}\cdot\xi=0$ (here we use that $n\geq 3$).
We define $\zeta_{1}$, $\zeta_{2}\in \mathbb{C}^{n}$ by
\begin{equation}\label{Zeta_1,2}
\begin{aligned}
 &\zeta_1=\frac{ih\xi}{2}+i\sqrt{1-h^2\frac{|\xi|^2}{4}}\mu^{(1)}+\mu^{(2)},\\
 &\zeta_2=-\frac{ih\xi}{2}+i\sqrt{1-h^2\frac{|\xi|^2}{4}}\mu^{(1)}-\mu^{(2)}.
 \end{aligned}
\end{equation}
Using Lemma \ref{Lemma} we write the CGO solutions to the equations $\mathcal{L}(x, D)\mathcal{U} =0$ and $\widetilde{\mathcal{L}}^{*}(x, D)\mathcal{V}=0$ in $O$ respectively
\begin{equation}\label{CGO-ON_O}
\begin{aligned}
\mathcal{U}(x ; h)= &e^{\frac{\zeta_{1}\cdot x}{h}}\left(a_{0}(x)+h a_{1}(x)+h^{2} a_{2}(x)\dots +h^{m-1}a_{m-1}(x)+r(x ; h)\right)= e^{\frac{\zeta_{1}\cdot x}{h}}(\mathrm{a}(x;h)), \\
\mathcal{V}(x ; h) =&e^{\frac{\zeta_{2}\cdot x}{h}}\left(b_{0}(x)+h b_{1}(x)+h^{2} b_{2}(x)\dots +h^{m-1}b_{m-1}(x)+\widetilde{r}(x ; h)\right)= e^{\frac{\zeta_{2}\cdot x}{h}}(\mathrm{b}(x;h)).
\end{aligned}
\end{equation}
We have the following decay estimates for the remainder terms
\begin{equation}\label{Decay_remainder_terms}
\|r\|_{H^{2m}_{\mathrm{scl}(O)}} = \mathcal{O}\left(h^{m}\right) \text { and } \|\widetilde{r}\|_{H^{2m}_{\mathrm{scl}(O)}} = \mathcal{O}\left(h^{m}\right).
\end{equation}
Following Proposition \ref{Existence of special solutions}, we write down the special solutions: 
\begin{align}\label{CGO SOLUTION 1}
      U(x ; h)=e^{\frac{\zeta_{1} \cdot x}{h}}\left(\mathrm{a}(x; h)\right) - e^{\frac{\zeta_{1}\cdot x^{*}}{h}}\left(\mathrm{a}(x^{*}; h)\right),
 \end{align}
 \begin{equation}
 \begin{aligned}\label{CGO SOLUTION 2}
 V(x ; h)=e^{\frac{\zeta_{2}\cdot x}{h}}\left(\mathrm{b}(x; h)\right) - e^{\frac{\zeta_{2} \cdot x^{*}}{h}}\left(\mathrm{b}(x^{*}; h)\right),
\end{aligned}
\end{equation}
where $ U(x;h) \text{ and } V(x;h)$ solves the polyharmonic equations $\mathcal{L}(x;D)U=0$, $\widetilde{\mathcal{L}}^{*}( x, D)V=0$ in $\Omega$, with the required boundary conditions $(-\Delta)^{k-1}U|_{\Gamma_{0}} = 0 = (-\Delta)^{k-1}V|_{\Gamma_{0}} \text{ for } 1\leq k \leq m$.

\subsection{The integral identity }  
We recall our given operator \eqref{Operator}
\[
 \mathcal{L}(x, D)= (-\Delta)^m + \sum_{j=0}^{m-1} \sum_{i_{1}, \cdots, i_{m-j}=1}^{n} A_{i_{1} \cdots i_{m-j}}^{m-j}(x) D_{i_{1} \cdots i_{m-j}}^{m-j}+q(x),
\]
where $A^{j}\in C_c^{\infty}(\Omega) \text { for } 1\leq j\leq m$ and $q \in L^{\infty}({\Omega})$. By using integration by parts, we get
\begin{equation}\label{Int_Identity}
\begin{aligned}
\int_{\Omega} &\left(\mathcal{L}(x,D)u\right)\overline{v}\, \D x - \int_{\Omega} u\, \overline{\mathcal{L}^{*}(x,D)v}\,\mathrm{d} x\\
=&\sum_{l=1}^{m}\int_{\partial\Omega} \left((-\Delta)^{(m-l)}u\overline {\left(\partial_{\nu} (-\Delta)^{l-1}v\right)}-
 (\partial_{\nu} (-\Delta)^{(m-l)}u) \overline{\left((-\Delta)^{l-1}v\right)}\right)\,\mathrm{d} S\\
&+\mathrm{i}\sum_{j=1}^{m}\sum_{l=1}^{j}\sum_{i_1,\cdots,i_{j}=1}^{n} \int_{\partial\Omega} (-1)^l \nu_{i_{j-l+1}}\left(D^{j-l}_{i_1 \dots i_{j-l}}u\right) D^{l-1}_{i_{j-l+2} \dots i_{j}}\left(A^{j}_{i_1 \dots i_{j}} \overline{v}\right)\, \mathrm{d} S,
\end{aligned}
\end{equation}
where $\D S$ is the surface measure on $\partial \Omega$ and for all $u, v \in H^{2m}(\Omega)$.
 
Let us take $u = U - \widetilde {U}$ where $U\in H^{2 m}(\Omega)$ solves
\begin{equation*}
\begin{aligned}
\mathcal{L}(x, D) U = 0, \quad \text { in } \Omega, \quad
\mbox{with }(-\Delta)^{k} U|_{\Gamma_0} = 0,
\quad \mbox{for } k=0,1, \ldots, m-1;
\end{aligned}
\end{equation*}
and $\widetilde{U}$ solves
\begin{equation}\label{Eq_1}
\widetilde{\mathcal{L}}(x, D) \widetilde{U}=0 \quad \mbox{in } \Omega, \quad \mbox{with } (-\Delta)^{k} \widetilde{U}|_{\partial\Omega} = (-\Delta)^{k} U|_{\partial\Omega}, \quad \mbox{for } k=0,1, \ldots, m-1.
\end{equation}
From the assumptions of Theorem \ref{Main_Result} on the boundary measurements, we must have 
\begin{equation}\label{Eq_2}
\partial_{\nu}(-\Delta)^{k} U|_{\Gamma}=\partial_{\nu}(-\Delta)^{k} \widetilde{U}|_{\Gamma}, 
\quad \mbox{for } k=0, 1, 2, \ldots, m-1.
\end{equation}
We choose $V \in H^{2m}(\Omega)$ to be any solution satisfying 
\begin{equation}\label{Eq_3}
\widetilde{\mathcal{L}}^{*}(x, D) V=0 \text{ in } \Omega, \quad \mbox{with }  (-\Delta)^{k}V|_{\Gamma_{0}} = 0, \quad \mbox{for } k =0,1 \cdots m-1.
\end{equation}
From \eqref{Eq_1}, \eqref{Eq_2} and \eqref{Eq_3}, using a similar argument as in \eqref{Int_Identity} we obtain
\begin{equation}\label{Int_Iden}
\begin{aligned}
0= \int_{\Omega} \left(\widetilde{\mathcal{L}}(x,D)(U-\widetilde{U})\right)\overline{V} \D x - \int_{\Omega} (U-\widetilde{U})\overline{\widetilde{\mathcal{L}}^{*}(x,D)V}\D x 
= \int_{\Omega} \left(\widetilde{\mathcal{L}}(x,D)(U-\widetilde{U})\right)\overline{V} \D x.
\end{aligned}
\end{equation}
A straightforward calculation entails
\begin{equation*}
\begin{aligned}
 \widetilde{\mathcal{L}}(x, D)(U-\widetilde{U})
 =&(\widetilde{\mathcal{L}}(x, D)-\mathcal{L}(x, D)) U\\
 =&
\sum_{j=0}^{m-1} \sum_{i_{1}, \cdots, i_{m-j}=1}^{n}\left( \widetilde{A}_{i_{1} \cdots i_{m-j}}^{m-j}- A_{i_{1} \cdots i_{m-j}}^{m-j} \right) D_{i_{1} \cdots i_{m-j}}^{m-j}U+(\widetilde{q}-q) U.
\end{aligned}
\end{equation*}
Therefore, the identity \eqref{Int_Iden} implies,
\begin{equation}\label{INTEGRAL IDENTITY}
\begin{aligned}
0=\sum_{j=0}^{m-1} \sum_{i_{1}, \cdots, i_{m-j}=1}^{n}\int_{\Omega}\left( \widetilde{A}_{i_{1} \cdots i_{m-j}}^{m-j}- A_{i_{1} \cdots i_{m-j}}^{m-j} \right) D_{i_{1} \cdots i_{m-j}}^{m-j}U\overline{V} \mathrm{d} x  + \int_{\Omega}(\widetilde{q}-q) U \overline{V} \mathrm{d} x.
\end{aligned}
\end{equation}
For simplicity, we denote $W^{j}= \widetilde{A}^{j} - A^{j}$ for $ j=1,2,\ldots,m \text { and } W^{0}= \widetilde{q}-q$.
Note that, using Proposition \ref{Existence of special solutions}, we can choose $U$ and $V$ to be of the form \eqref{CGO SOLUTION 1}, \eqref{CGO SOLUTION 2}.
We compute
\begin{equation}\label{Product of exponential}
\begin{aligned}
e^{x\cdot\zeta_{1}/h}e^{x\cdot \overline{\zeta}_{2}/h}\,=\,&e^{ix\cdot\xi} \\ \vspace{1ex}
e^{x\cdot\zeta_{1}/h}e^{(x^{*})\cdot\overline{\zeta}_{2}/h}\,=\,&e^{ix\cdot \xi_{+}+2\mu^{(2)}_{n}x_{n}/h}\\ 
\vspace{1ex}
e^{(x^{*})\cdot\zeta_{1}/h}e^{x\cdot\overline{\zeta}_{2}/h}\,=\,&e^{ix\cdot \xi_{-}-2\mu^{(2)}_{n}x_{n}/h} \\ 
\vspace{1ex}
e^{(x^{*})\cdot\zeta_{1}/h}e^{(x^{*})\cdot\overline{\zeta}_{2}/h}\,=\,&
e^{i(x^{*})\cdot\xi}
\end{aligned}
\end{equation}
where $\xi_{+}$ , $\xi_{-} \in \R^{n}$ are 
\[\xi_{+} =\left(\xi',+\frac{2}{h}\sqrt{1-h^{2}\frac{|\xi|^{2}}{4}}\mu^{(1)}_{n}\right),
\qquad
\xi_{-} =\left(\xi',-\frac{2}{h}\sqrt{1-h^{2}\frac{|\xi|^{2}}{4}}\mu^{(1)}_{n}\right),\quad \mbox{where }\xi'=(\xi_1,\cdots,\xi_{n-1}).
\]

To eliminate the unwanted terms, we further assume that $\mu^{(2)}_{n}=0$ and $\mu^{(1)}_{n}\neq 0$. 
So, by the Riemann--Lebesgue lemma and the fact that $|\xi_{+}|,|\xi_{-}|\rightarrow\infty$ as $h\rightarrow 0$, we have the following four limiting identities,
\begin{equation}\label{vanish_1}
\lim_{h\to 0} \int_{\Omega}(W^{0})e^{(x^{*})\cdot\zeta_{1}/h}e^{x\cdot\overline{\zeta}_{2}/h}\mathrm{a}(x^{*} ; h) \overline{\mathrm{b}(x ; h)} \mathrm{d}x
=0.
\end{equation}
\begin{equation}\label{vanish_2}
\begin{aligned}
\lim_{h\to 0} \sum_{j = 0}^{m-1}\sum_{i_{1},\cdots, i_{m-j}=1}^{n} h^{m-j}\int_{\Omega}&\left(W_{i_{1}\cdots i_m-j}^{m-j}\right)e^{(x^{*})\cdot\zeta_{1}/h}e^{x\cdot\overline{\zeta}_{2}/h} \\
&\times\left(\prod_{k=1}^{m-j}\left(\frac{-\mathrm{i}}{h} (\mu^{(2)}+\mathrm{i}\mu^{(1)})_{i_{k}}+D_{i_{k}}\right)\mathrm{a}(x^{*} ; h)\right)\overline{\mathrm{b}(x ; h)} \mathrm{d}x
=0.
\end{aligned}
\end{equation}
\begin{equation}\label{vanish_3}
\lim_{h\to 0} \int_{\Omega}(W^{0})e^{x\cdot\zeta_{1}/h}e^{(x^{*})\cdot\overline{\zeta}_{2}/h}\mathrm{a}(x ; h)\overline{\mathrm{b}(x^{*} ; h)} \mathrm{d}x
=0.
\end{equation}
\begin{equation}\label{vanish_4}
\begin{aligned}
\lim_{h\to 0} \sum_{j = 0}^{m-1}\sum_{i_{1},\cdots, i_{m-j}=1}^{n} h^{m-j}\int_{\Omega}&\left(W_{i_{1}\cdots i_m-j}^{m-j}\right)e^{x\cdot\zeta_{1}/h}e^{(x^{*})\cdot\overline{\zeta}_{2}/h}\\
&\times\left(\prod_{k=1}^{m-j}\left(\frac{-\mathrm{i}}{h} (\mu^{(2)}+\mathrm{i}\mu^{(1)})_{i_{k}}+D_{i_{k}}\right)\mathrm{a}(x ; h)\right)\overline{\mathrm{b}(x^{*} ; h)} \mathrm{d}x
=0.
\end{aligned}
\end{equation}

Now, inserting the CGO solutions \eqref{CGO SOLUTION 1} and \eqref{CGO SOLUTION 2} into the identity \eqref{INTEGRAL IDENTITY} and using \eqref{Product of exponential} we get
\begin{equation*}
\begin{aligned}
0=&\sum_{j = 0}^{m-1}\sum_{i_{1},\cdots, i_{m-j} =1}^{n} \int_{\Omega}\left(W_{i_{1}\cdots i_{m-j}}^{m-j}\right)\prod_{k=1}^{m-j}\left(\frac{-\mathrm{i}}{h} (\mu^{(2)}+\mathrm{i}\mu^{(1)})_{i_{k}}+D_{i_{k}}\right)\mathrm{a}(x ; h)\overline{\mathrm{b}(x ; h)} e^{ix\cdot\xi}\mathrm{d}x\\
&+\int_{\Omega}(W^{0}) e^{ix\cdot\xi} \mathrm{a}(x ; h) \overline{\mathrm{b}(x ; h)} \mathrm{d} x\\
&-\sum_{j = 0}^{m-1}\sum_{i_{1},\cdots, i_{m-j}=1}^{n}\int_{\Omega}\left(W_{i_{1}\cdots i_m-j}^{m-j}\right)e^{x\cdot\zeta_{1}/h}e^{(x^{*})\cdot\overline{\zeta}_{2}/h}\prod_{k=1}^{m-j}\left(\frac{-\mathrm{i}}{h} (\mu^{(2)}+i\mu^{(1)})_{i_{k}}+D_{i_{k}}\right)\mathrm{a}(x ; h) \\
&\hspace{3cm}\times\overline{\mathrm{b}(x^{*} ; h)} \mathrm{d}x 
- \int_{\Omega}(W^{0})e^{x\cdot\zeta_{1}/h}e^{(x^{*})\cdot\overline{\zeta}_{2}/h}\mathrm{a}(x^{*} ; h) \overline{\mathrm{b}(x ; h)} \mathrm{d}x \\
&-\sum_{j = 0}^{m-1}\sum_{i_{1},\cdots, i_{m-j}=1}^{n}\int_{\Omega}\left(W_{i_{1}\cdots i_m-j}^{m-j}\right)e^{(x^{*})\cdot\zeta_{1}/h}e^{x\cdot\overline{\zeta}_{2}/h} \\
&\hspace{3cm}\times\left(\prod_{k=1}^{m-j}\left(\frac{-\mathrm{i}}{h} (\mu^{(2)}+ i\mu^{(1)})_{i_{k}}+D_{i_{k}}\right)\mathrm{a}(x^{*} ; h)\right)
\overline{\mathrm{b}(x ; h)} \mathrm{d}x \\
&- \int_{\Omega}(W^{0})e^{(x^{*})\cdot\zeta_{1}/h}e^{x\cdot\overline{\zeta}_{2}/h}\mathrm{a}(x ; h)\overline{\mathrm{b}(x^{*} ; h)} \mathrm{d}x \\
&+\sum_{j = 0}^{m-1}\sum_{i_{1},\cdots, i_{m-j} =1}^{n} \int_{\Omega}\left(W_{i_{1}\cdots i_{m-j}}^{m-j}\right)\prod_{k=1}^{m-j}\left(\frac{-\mathrm{i}}{h} (\mu^{(2)}+\mathrm{i}\mu^{(1)})_{i_{k}}+D_{i_{k}}\right)\mathrm{a}(x^{*} ; h)\\
&\hspace{3cm} \times\overline{\mathrm{b}(x^{*} ; h)}e^{i(x^{*})\cdot\xi} \mathrm{d}x
+\int_{\Omega}(W^{0})e^{i(x^{*})\cdot\xi} \mathrm{a}(x^{*} ; h) \overline{\mathrm{b}(x^{*} ; h)} \mathrm{d}x.
\end{aligned}
\end{equation*}
Making the change of variables $x\mapsto x^{*}$ we obtain the integral identity in $O$ as
\begin{equation}\label{Integral Identity Main}
\begin{aligned}
0=&\sum_{j = 0}^{m-1}\sum_{i_{1},\cdots, i_{m-j} =1}^{n} \int_{O}\left(W_{i_{1}\cdots i_{m-j}}^{m-j}\right)\left(\prod_{k=1}^{m-j}\left(\frac{-\mathrm{i}}{h} (\mu^{(2)}+\mathrm{i}\mu^{(1)})_{i_{k}}+D_{i_{k}}\right)\mathrm{a}(x ; h)\right)\\
&\times\overline{\mathrm{b}(x ; h)}e^{ix\cdot\xi} \mathrm{d}x-\sum_{j = 0}^{m-1}\sum_{i_{1},\cdots, i_{m-j}=1}^{n}\int_{O}\left(W_{i_{1}\cdots i_m-j}^{m-j}\right)e^{x\cdot\zeta_{1}/h}e^{(x^{*})\cdot\overline{\zeta}_{2}/h}\\
&\times\left(\prod_{k=1}^{m-j}\left(\frac{-\mathrm{i}}{h} (\mu^{(2)}+i\mu^{(1)})_{i_{k}}+D_{i_{k}}\right)\mathrm{a}(x ; h)\right)\overline{\mathrm{b}(x^{*} ; h)} \mathrm{d}x\\
&- \int_{O}(W^{0})e^{(x^{*})\cdot\zeta_{1}/h}e^{x\cdot\overline{\zeta}_{2}/h}\mathrm{a}(x^{*} ; h) \overline{\mathrm{b}(x ; h)} \mathrm{d}x +\int_{O}(W^{0})e^{i x\cdot \xi} \mathrm{a}(x ; h) \overline{\mathrm{b}(x ; h)} \mathrm{d} x,
\end{aligned}
\end{equation}
for all $\xi,\mu^{(1)},\mu^{(2)}\in\mathbb{R}^{n}$ with
\begin{equation}\label{Ortho_condition}
\mu^{(1)}\cdot\mu^{(2)}=\xi\cdot\mu^{(1)}=\xi\cdot\mu^{(2)}=0,\quad |\mu^{(1)}|=|\mu^{(2)}|=1, \quad \mu^{(2)}_{n}=0, \quad \mu^{(1)}_{n}\neq 0.
\end{equation}
Recall the form of $a(x;h)$ and $b(x;h)$ from \eqref{CGO SOLUTION 1}, \eqref{CGO SOLUTION 2}.
Here $ a_{0}(x), b_{0}(x) \in C^{\infty}(\overline{O}, \mathbb{C})$ satisfy the transport equations
\begin{equation}\label{Transport Eqn}
\left( 2(\mu^{(2)}+i\mu^{(1)})\cdot \nabla \right)^{m} a_{0}(x) = \left( 2(-\mu^{(2)}+i\mu^{(1)})\cdot \nabla \right)^{m} b_{0}(x) = 0 \text { in } O \text { for } m\geq 2. 
\end{equation}
Next, we solve the transport equations \eqref{Transport Eqn} following \cite{bhattacharyya2021unique}. 
Observe that, up to a rotation, we can take $\mu^{(1)}=e_n$ and $\mu^{(2)}=e_j$, $j\neq n$ where $\{e_1,\cdots,e_n\}$ is the standard orthonormal basis for $\R^n$.
Then the transport equations \eqref{Transport Eqn} become
\begin{equation}
2^m\left(\partial_{x_j}+ \mathrm{i} \partial_{x_n}\right)^m a_0=0 \text{ and }  2^m\left(\partial_{x_j}+ \mathrm{i} \partial_{x_n}\right)^m \overline{b_0}=0 \quad \text { in } O \text { for } m\geq 2.
\end{equation}
For $1\leq j \leq n-1$, denoting $z=x_j + \mathrm{i} x_n$ the above transport equations reduce to 
\begin{equation}\label{Trans Eqn new}
\partial_{\bar z}^m a_0=0 \quad  \text{ and } \quad \partial_{\bar z}^{m}\overline{b_0} = 0.
\end{equation}
The general solution of equation \eqref{Trans Eqn new} can be given by
\begin{equation}
    a_0(x)=g(x'')\sum_{k=0}^{m-1}(z-\bar{z})^k f_k(z),
\end{equation}
where $f_k$ is a holomorphic function for all $0 \leq k \leq m-1$, $g$ is any smooth function and $x''=x-x_je_j - x_ne_n$.

\subsection{Recovering the coefficients}\label{Subsec_Coeff_recovery}
In this section, we closely follow \cite{bhattacharyya2021unique} to recover the coefficients from the integral equation \eqref{Integral Identity Main}. We multiply \eqref{Integral Identity Main} by suitable powers of $h$ and take $h\to 0$ to obtain momentum ray transforms of the coefficients. We recall the definition of partially isotropic tensors. We say a tensor $W^j$, $j\geq 2$ is partially isotropic if there exists a symmetric $(j-2)$-tensor field  $\overline{W}^{j-2}$ in $O$ such that $W^{j}=i_{\delta}\overline{W}^{j-2}=\overline{W}^{j-2} \odot \delta$, where $\delta$ is the Kronecker delta tensor and $\odot$ denotes the symmetric product of tensors. We assume that $A^{m}=i_{\delta} \overline{A}^{m-2}$ (respectively for $\widetilde{A}^{m}$ ) and $\operatorname{denote} \overline{W}^{m-2}= \overline{\widetilde{A}}^{m-2}- \overline{A}^{m-2}$. 
 
 \begin{proof}[Proof of Theorem \ref{Main_Result}]
 \textbf{Step 1.}
We start with a rotation of coordinates in \eqref{Integral Identity Main} by taking  $e_n$ to $e_1$ and denote $\eta$ as any unit vector perpendicular to $e_1$.
So, denoting the new coordinates by $y$ (w.r.t. some orthonormal basis $\{e_1,\eta,\cdots\}$), we get $z=y_1+\mathrm{i}(y\cdot\eta)=y_1+\mathrm{i}y_2$ and $y=(y_1,y') = (y_1,y_2, y'')$.
We multiply \eqref{Integral Identity Main} by $h^{m-1}$ and let $h\to 0$ along with \eqref{vanish_1}, \eqref{vanish_2},
 \eqref{vanish_3}, \eqref{vanish_4}, and \eqref{Decay_remainder_terms} we have
\begin{equation}\label{Int_id_h2}
\begin{aligned}
&\sum_{i_{1},\dots ,i_{m-2} =1}^{n} \int_{O}\overline{W}_{i_{1}\cdots ,i_{m-2}}^{m-2}\left(\prod_{k =1}^{m-2}(e_1+\mathrm{i}\eta)_{i_{k}} \right)\left[-T a_{0}(y)\right] \overline{b_{0}(y)}\mathrm{~d} y\\
&+\sum_{i_{1} , \dots, i_{m-1}=1}^{n} \int_{O} W_{i_{1} \cdots i_{m-1}}^{m-1} \left(\prod_{k =1}^{m-1} (e_{1}+\mathrm{i}\eta)_{i_{k}}\right)a_{0}(y)  \overline{b_{0}(y)}\mathrm{~d} y=0.
\end{aligned}
\end{equation}
We choose the amplitudes $ a_{0}(y) \text{ and } b_{0}(y)$ as 
\begin{equation*}
    a_{0}(y) = 1 \mbox { and }  \overline{b_{0}(y)} = (z-\bar{z})^{k}f(z) g(y^{\prime\prime}), \quad \text { for all } k \text { with } 0\leq k\leq m-1,
\end{equation*}
for any holomorphic function $f(z)$ and any smooth function $g(y^{\prime\prime})$.
For $\lambda \in \R$ we choose
\[ f(z) = e^{-\mathrm{i}\lambda z}. \]
Note that, using compact support of $W^j$'s we can extend $W^{m-1}(\cdot, \cdot, y^{\prime\prime})=0$ over $\mathbb{R}^{2} \backslash O$,  therefore, for each $0\leq k \leq m-1$ we obtain, 
\begin{equation*}
\int_{\mathbb{R}^{n-2}}\left( \int_{\mathbb{R}^{2}} \sum_{i_{1}, \cdots , i_{m-1}=1}^{n} W_{i_{1} \cdots i_{m-1}}^{m-1} \left(\prod_{k =1}^{m-1}(e_{1}+\mathrm{i}\eta)_{i_{k}}\right)y_2^{k} e^{-\mathrm{i}\lambda z}\mathrm{~d}y_{1}\mathrm{~d}y_2\right)g(y^{\prime\prime}) \mathrm{~d} y^{\prime\prime}=0.
\end{equation*}
Varying $g$ to be any smooth function in the $y''$ variable, we obtain
\begin{equation*}
  \int_{\mathbb{R}^{2}} \sum_{i_{1}, \cdots , i_{m-1}=1}^{n} W_{i_{1} \cdots i_{m-1}}^{m-1} \left(\prod_{k =1}^{m-1}(e_{1}+\mathrm{i}\eta)_{i_{k}}\right) y_2^{k} e^{-\mathrm{i}\lambda z}\mathrm{~d}y_{1}\mathrm{~d}y_2= 0.
\end{equation*}
After performing the integration in the $y_{1}$ variable we obtain
\begin{equation}\label{m-1 MRT}
     \int_{\mathbb{R}} \sum_{i_{1}, \cdots , i_{m-1}=1}^{n}\widehat{ W}_{i_{1} \cdots i_{m-1}}^{m-1}\left(\prod_{k =1}^{m-1}(e_{1}+\mathrm{i}\eta)_{i_{k}}\right) y_2^{k} e^{\lambda y_2}\mathrm{~d}y_2 =0, \quad\text{ for a.e. }y^{\prime\prime}\in \mathbb{R}^{n-2}  \text{ and for all } \lambda \in \mathbb{R},
\end{equation}
where we denote $\widehat{.}$
to be the partial Fourier transform in the $y_{1}$ variable. We set $\lambda = 0$ in above to obtain
\begin{equation*}
     \int_{\mathbb{R}} \sum_{i_{1}, \cdots , i_{m-1}=1}^{n}\widehat{ W}_{i_{1} \cdots i_{m-1}}^{m-1}\left(\prod_{k =1}^{m-1}(e_{1}+\mathrm{i}\eta)_{i_{k}}\right) y_2^{k} \mathrm{~d}y_2=0.
\end{equation*}
Since the above identity can be interpreted as the vanishing of a certain MRT, by using Lemma \ref{Lemma_MRT}, we see that there exists a
symmetric $(m-3)$-tensor field $\widehat{V}^{m-1, 1,0}(0,y^{\prime})$, compactly supported in $y^{\prime}$ such that 
\[
\widehat{W}^{m-1}(0,y^{\prime}) = 
\begin{cases}
i_{\delta}\widehat{V}^{m-1, 1,0}(0,y^{\prime}), \quad &\mbox{when } m\geq3\\
0 & \mbox{when } m=2
\end{cases},
\quad \mbox{for a.e. }  y^{\prime}\in \R^{n+1}.
\]
Next, for $m\geq3$, we differentiate \eqref{m-1 MRT} with respect to $\lambda$ and set $\lambda = 0$ to get
\begin{equation}\label{MRT Con_1}
    \begin{aligned}
         0=&\int_{\mathbb{R}}  \sum_{i_{1}, \cdots , i_{m-1}=1}^{n}\left(\widehat{ W}_{i_{1} \cdots i_{m-1}}^{m-1}(0, y^{\prime})\prod_{k =1}^{m-1}(e_{1} + \mathrm{i}\eta)_{i_{k}}\right)y_2^{k+1} \mathrm{~d}y_2\\
    &+ \sum_{i_{1}, \cdots , i_{m-1}=1}^{n}\int_{\mathbb{R}} y_2^{k}\frac{\mathrm{d}}{\mathrm{d}\lambda}\left(\widehat{ W}_{i_{1} \cdots i_{m-1}}^{m-1}(0, y^{\prime})\prod_{k =1}^{m-1}(e_{1} + \mathrm{i}\eta)_{i_{k}}\right) \mathrm{~d}y_2= 0.
    \end{aligned}
\end{equation}
Since we have $\widehat{W}^{m-1}(0, y^{\prime}) = i_{\delta}\widehat{V}^{m-1, 1, 0}(0, y^{\prime})$, 
so the first term in \eqref{MRT Con_1} vanishes using the fact that $(e_{1}+\mathrm{i}\eta)\cdot(e_{1}+\mathrm{i}\eta) = 0.$
Now, repeating the same steps as before, we see that there exists a symmetric $(m-3)$-tensor field $\widehat{V}^{m-1, 1, 1}(0, y^{\prime})$ such that 
$$
\left.\frac{\mathrm{d}}{\mathrm{d} \lambda}\right|_{\lambda=0} \widehat{W}^{m-1}(0, y^{\prime})=i_{\delta} \widehat{V}^{m-1, 1, 1}(0, y^{\prime}), \quad \text { for a.e. } y^{\prime} \text {. }
$$
Repeating the same technique we obtain
\begin{equation}\label{rth_ derivative of m-1 tensor}
\left.\frac{\mathrm{d}^{r}}{\mathrm{d} \lambda^{r}}\right|_{\lambda=0} \widehat{W}^{m-1}(0, y^{\prime})=i_{\delta} \widehat{V}^{m-1, 1, r}(0, y^{\prime}) \quad \text { for a.e. } y^{\prime}.
\end{equation}
Since $W^{m-1}$ is compactly supported in $\Omega$, we have that $\widehat{W}^{m-1}$ is analytic in $\lambda$ variable. Using Payley--Weiner's theorem, we have
\[
\widehat{W}^{m-1}(\lambda, y^{\prime}) = i_{\delta}\widehat{V}^{m-1, 1}(\lambda, y^{\prime}) \text{ a.e. in } \Omega, \quad m\geq 3,
\]
where $\widehat{V}^{m-1, 1}$ is symmetric $(m-3)$-tensor field and
\[\widehat{V}^{m-1,1}(\lambda, y^{\prime}) = \sum_{r=0}^{\infty}\frac{{\lambda}^{r}}{r!}\widehat{V}^{m-1,1,r}(0, y^{\prime}).
\]
Hence, by using the inverse Fourier transform, we get
\begin{equation}\label{partial isotropic m-1}
W^{m-1}(y_{1}, y^{\prime}) =
\begin{cases}
i_{\delta}V^{m-1, 1}(y_{1}, y^{\prime})\quad &\mbox{for }m\geq 3,\\
0 \quad &\mbox{for } m=2.
\end{cases}
\end{equation}
By using \eqref{partial isotropic m-1} in \eqref{Int_id_h2}, we obtain
\begin{align}\label{m-2 MRT}
    \sum_{i_{1},\dots ,i_{m-2} =1}^{n} \int_{\mathbb{R}} \overline{W}_{i_{1}\cdots i_{m-2}}^{m-2}\left(\prod_{k=1}^{m-2}(e_{1}+\mathrm{i}\eta)_{i_{k}}\right) \left[-T a_{0}(y)\right]\overline{b_{0}(y)} \mathrm{~d}y = 0.
\end{align}
Note that, using compact support of $W^j$'s we can smoothly extend $W^{m-2}(\cdot, \cdot, y^{\prime\prime})=0$ over $\mathbb{R}^{2}$.
Let us now choose 
 $$
 a_{0}(y) = (z-\bar{z})\quad  \text{ then } \quad Ta_{0}(y) = -1.
 $$
By choosing $\overline{b_{0}(y)} = (z-\bar{z})^{k}e^{-\mathrm{i}\lambda z}g(y^{\prime\prime})$ for all $0\leq k \leq m-2$ and $g$ is any smooth function in $y^{\prime\prime}\in\mathbb{R}^{n-2}$, we are in the same set-up as before. Repeating the same steps as above, we see that 
\begin{equation}\label{Step_Cond_1}
   \overline{W}^{m-2}\left(y_{1}, y^{\prime}\right)=
   \begin{cases}
    i_{\delta} V^{m, 2}\left(y_{1}, y^{\prime}\right) \quad &\mbox{for }m\geq 4\\
    0 &\mbox{for }m = 2, 3
   \end{cases},
\end{equation}
where $V^{m,2}$ is a symmetric, smooth, compactly supported, $(m-4)$-tensor field.
Since we have $W^{m}= i_{\delta} \overline{W}^{m-2}$, therefore, by using \eqref{Step_Cond_1} we obtain 
\begin{equation}\label{Induction_intial}
W^{m}= 
\begin{cases}
i_{\delta}^{2} V^{m, 2} \quad &\mbox{if }m\geq 4,\\
0 &\mbox{if } m=2,3,
\end{cases} 
\quad\mbox{and}\quad 
W^{m-1}=
\begin{cases}
i_{\delta} V^{m-1,1}, \quad &\mbox{for } m\geq 3,\\
0 & \mbox{for } m=2.
\end{cases}
\end{equation}
The rest of the proof follows from the mathematical induction given below.

\smallskip

\textbf{Step 2.}
We assume, for $j\geq 1$, \begin{equation}\label{Induction step}
W^{m-r} = \begin{cases}
   i_{\delta}^{j+1-r} V^{m-r, j+1-r}, \quad &\mbox{for } 2j+2-m\leq r \leq j,\\
   0 &\mbox{for } 0\leq r < 2j+2-m,
\end{cases}
\end{equation}
for some compactly supported, smooth, symmetric tensor fields 
$V^{m-r, j+1-r}$ of order
$m+r-2j-2$. For $j=1$, the initial induction step \eqref{Induction step} follows from \eqref{Induction_intial}.
We want to show that \eqref{Induction step} holds for $j+1$, assuming it holds for $j$. That is, assuming \eqref{Induction step}, we prove
\begin{equation}\label{final_induction}
W^{m-r} = \begin{cases}
    i_{\delta}^{j+2-r} V^{m-r, j+2-r}, \quad &\mbox{for } 2j+4-m \leq r \leq j+1,\\
    0 &\mbox{for } 0\leq r < 2j+4-m.
\end{cases}
\end{equation}

We multiply \eqref{Integral Identity Main} by $h^{m-j-1}$ and assuming induction hypothesis \eqref{Induction step} along with \eqref{vanish_2}, \eqref{vanish_3} 
 we obtain
\begin{equation}\label{In_id_j+2}
\begin{aligned}
0=\sum_{r=0}^{j+1}\sum_{i_1,\cdots,i_{m-2j-2+r}=1}^{n} \int_{O} V^{m-r,j+1-r}_{i_1\dots i_{m-2j-2+r}} \left(\prod_{k=1}^{m-2j-2+r}(e_{1}+\mathrm{i}\eta)_{i_{k}}\right)\left[(-T)^{j+1-r}a_0(y)\right]\,\overline{b_0(y)} \mathrm{~d}y,
\end{aligned}
\end{equation}
as $h \to 0$.
Now we choose
$$
a_{0}(y) = 1
 \mbox{ and } \overline{b_{0}(y)} = (z-\bar{z})^{k}e^{-\mathrm{i}\lambda z} g(y^{\prime\prime})\quad  \text{ for all } 0\leq k \leq m ,
 $$ 
 where $ g$ is an arbitrary smooth function in $y''\in \mathbb{R}^{n-2}$ and $\lambda \in \mathbb{R}$. We notice that 
 $$
 T^{j+1-r} a_{0}(y) = 0 \quad \text { for all } \quad 0\leq r\leq j.
 $$
Therefore, we get,
 \begin{equation*}
\begin{aligned}
0=\sum_{i_1,\cdots,i_{m-j-1}=1}^{n} \int_{O} V^{m-j-1, 0}_{i_1\dots i_{m-j-1}} \left(\prod_{k=1}^{m-j-1}(e_{1}+\mathrm{i}\eta)_{i_{k}}\right)y_2^{k}e^{-\mathrm{i}\lambda z} g(y^{\prime\prime}) \mathrm{~d}y_2,
\end{aligned}
\end{equation*}
for all $0\leq k \leq m $ and $\lambda \in \mathbb{R}.$ Now, proceeding similarly as we have done in step 1, we have
\begin{equation}\label{Induction_1}
    W^{m-j-1}=i_{\delta}V^{m-j-1,1}.
\end{equation}
Note that, using \eqref{Induction_1} in \eqref{In_id_j+2} we obtain
\begin{equation}\label{eq_1}
    \begin{aligned}
0=\sum_{r=0}^{j}\sum_{i_1,\cdots,i_{m-2j-2+r}=1}^{n} \int_{O} V^{m-r,j+1-r}_{i_1\dots i_{m-2j-2+r}} \left(\prod_{k=1}^{m-2j-2+r}(e_{1}+\mathrm{i}\eta)_{i_{k}}\right)\left[(-T)^{j+1-r}a_0(y)\right]\,\overline{b_0(y)}
    \, \mathrm{~d}y.
\end{aligned}
\end{equation}
Now, we choose $a_{0}(y) = (z-\bar{z}) 
 \mbox{ and } \overline{b_{0}(y)} = (z-\bar{z})^{k}e^{-\mathrm{i}\lambda z} g(y^{\prime\prime}), \text { for all } 0\leq k \leq m$ where $ g$ is an arbitrary smooth function in $y''$ variables and $\lambda \in \mathbb{R}$. Simplifying \eqref{eq_1} we get,
 \begin{equation*}
\begin{aligned}
0=\sum_{i_1,\cdots,i_{m-j-2}=1}^{n} \int_{O} V^{m-j, 1}_{i_1\dots i_{m-j-2}} \left(\prod_{k=1}^{m-j-2}(e_{1}+\mathrm{i}\eta)_{i_{k}}\right)y_2^{k+1}e^{-\mathrm{i}\lambda z} g(y^{\prime\prime})
    \,\mathrm{~d}y_2.
\end{aligned}
\end{equation*}
Again, we use the same technique as in Step 1; to obtain
\begin{equation}\label{Induction_2}
    V^{m-j,1}=
    \begin{cases}
     i_{\delta}V^{m-j,2} \quad &\mbox{for } m-j\geq 4.\\
     0 &\mbox{for } m-j<4.
    \end{cases}
\end{equation}
Since, for $m-j\geq 4$ we have $W^{m-j}=i_{\delta}V^{m-j,1}$ in the induction assumption, therefore, using \eqref{Induction_2} we get 
\[
W^{m-j} = i_{\delta}^{2}V^{m-j,2}, \quad \mbox{for }m-j\geq 4.
\]
Repeating this argument for $l$-times, for $0\leq l \leq j+1$ and using amplitudes
\[
a_0(y) = (z-\overline{z})^{l} \quad \mbox{and} \quad \overline{b_0(y)}= (z-\bar{z})^{k}e^{-\mathrm{i}\lambda z}g(y^{\prime\prime}).
\]
We obtain 
\[
W^{m-j+l-1} =
\begin{cases}
i_{\delta}^{l+1}V^{m-j+l-1,l+1} &\mbox{ for}\quad m-j-l\geq 3\\
0 & \mbox{ for}\quad m-j-l < 3
\end{cases}, \quad 0\leq l\leq j+1, \quad j\geq 1.
\]
The proof is now complete by observing that, setting $j$ large and $0\leq l\leq j+1$ suitably, we can make $m-j-l < 3$ for any fixed $m\geq2$. Hence, by mathematical induction, we get all the coefficients $W^{m-k}$, $0\leq k \leq m$ are zero, that is,
\[
A^{m-k} = \widetilde{A}^{m-k}, \quad \mbox{for } 0\leq k \leq m.
\]
\end{proof}

\section*{Acknowledgements}
The authors want to express their gratitude to the Indian Institute of Science Education and Research, Bhopal, and the Department of Mathematics at IISER-B for their invaluable support. S.B. was partially supported by the grant: SRG/2022/00129 from the Science and Engineering Research Board, India.

\bibliographystyle{alpha}
\bibliography{bibliography.bib}

\end{document}